\documentclass[11pt,a4paper]{amsart}

\usepackage[english]{babel} 
\usepackage[T1]{fontenc} 
\usepackage{amsmath}
\usepackage{amssymb} 
\usepackage{ucs} 
\usepackage[utf8x]{inputenc}
\usepackage[mathcal]{eucal} 
\usepackage[]{graphicx} 
\usepackage{latexsym}
\usepackage{amsthm} 
\usepackage{stmaryrd}
\usepackage{verbatim}
\usepackage{pgf,tikz}
\usetikzlibrary{arrows}

\usepackage{hyperref}

\newtheorem{defi}{Definition}[section] 
\newtheorem{theo}[defi]{Theorem}
\newtheorem{theorem}[defi]{Theorem}
\newtheorem{coro}[defi]{Corollary} 
\newtheorem{lemme}[defi]{Lemma}

\newtheorem{prop}[defi]{Proposition}

\theoremstyle{remark}

\newcommand{\R}{\mathbb{R}} 
\newcommand{\Q}{\mathbb{Q}}

\newcommand{\N}{\mathbb{N}} 

\newcommand{\e}{\varepsilon}
\newcommand{\p}{\varphi}

\newcommand{\PP}{\mathcal{P}}

\newcommand{\intoo}[2]{\mathopen{]}#1\,,#2\mathclose{[}}
\newcommand{\intff}[2]{\mathopen{[}#1\,,#2\mathclose{]}}

\newcommand{\intfo}[2]{\mathopen{[}#1\,,#2\mathclose{[}}

\renewcommand{\tilde}{\widetilde}

\title{Attractors in spacetimes and time functions} 
\author{Daniel Monclair}

\thanks{Project supported by the National Research Fund, Luxembourg.}

\begin{document}

\begin{abstract}
We develop a new approach to the existence of time functions on Lorentzian manifolds, based on Conley's work regarding Lyapunov functions for dynamical systems. We recover Hawking's result that a stably causal admits a time function through a more general result giving the existence of a continuous function that is non decreasing along all future directed causal curves, and increasing along such curves that lie outside a special region of the spacetime, called the chain recurrent set, which is empty for stably causal spacetimes.\\
The construction is based on a notion of attractive sets in spacetimes.
\end{abstract}

\maketitle

\setcounter{tocdepth}{1}
\tableofcontents
\section{Introduction}

This paper is devoted to the relationship between causality theory and dynamical systems, with an emphasis on using dynamical techniques to prove Lorentzian results. One of the most striking similarities is the notion of time functions for Lorentzian manifolds and Lyapunov functions for dynamical systems.\\
\indent Causality theory is the study of the relationship between points of a Lorentzian manifold that are linked by causal trajectories. One of the classic problems in causality theory is the existence of a (global) time function, i.e. a continuous real-valued function that increases along future directed causal curves. This is now well understood: the first existence result was Geroch's Theorem for globally hyperbolic spacetimes \cite{geroch}, then Hawking \cite{hawking} showed that the existence of a time function is equivalent to stable causality (i.e. no nearby metrics have closed causal curves). It is well known that there was a gap in Hawking's proof, as it uses the existence of a temporal function (a smooth function with timelike gradient) in order to assure stable causality. The smoothness problem was only solved in 2005 by Bernal and Sanchez \cite{BS_fonction_lisse}. Other successful techniques for the existence of temporal functions have since been developed (see \cite{fathi_siconolfi}, \cite{chrusciel}, \cite{muller_sanchez}, \cite{suhr}, \cite{minguzzi}).\\
\indent In classical dynamical systems, a Lyapunov function for a flow is a continuous function that is non increasing along orbits (physically, it can be thought of as the mechanical energy). In general, one does not look for a Lyapunov function that is decreasing along all orbits (as they only exist for very simple dynamics, which never occurs on a compact manifold). Instead, the goal is to understand the orbits on which such a function is necessarily constant, and to find an "optimal" Lyapunov function: one that is constant along these orbits, and decreasing along all others. This was achieved by Conley, who linked the existence of such a function to two important concepts in topological dynamics: chain recurrence and attractive sets.\\

\indent Our goal in this paper is to use techniques that are close to the ones developed by Conley \cite{conley} in order to construct time functions. We will define the chain recurrent set $R(g)\subset M$ of a Lorentzian manifold $(M,g)$ (see Definition \ref{def_chain_rec}), and get a result that is similar to Conley's.\\
\indent Given a point $p\in M$, we denote by $J^+(p)$ the \textbf{causal future} of $p$, i.e. the set of endpoints of future directed causal curves starting at $p$ (see section \ref{sec:causality} for basic definitions of causality theory).

\begin{theo} \label{chain_recurrent_partial_time} Let $(M,g)$ be a time oriented Lorentzian manifold. There is a continuous function $\tau :M\to\R$ such that:
\begin{itemize}
\item For all $p\in M$ and $q\in J^+(p)$, $f(q)\ge f(p)$,
\item For all $p\notin R(g)$ and $q\in J^+(p)\setminus\{p\}$, $f(q)>f(p)$.
\end{itemize}
\end{theo}

In the stably causal case, we get a new construction of time functions.

\begin{theo} \label{stable_no_chains} If $(M,g)$ is stably causal, then $R(g)=\emptyset$. \end{theo}

An interesting difference between this construction and the classical one using volumes of past sets is that the continuity of the time function does not depend on the causal structure, which is the case in Hawking's construction \cite{hawking} (which is itself a variation of the construction of Geroch \cite{geroch}).

\subsection{Lorentzian conformal classes as multi-valued dynamical systems}  If $(M,g)$ is a time oriented Lorentz manifold,  let $C(p)$ be the set of future oriented causal vectors tangent at $p\in M$ (these definitions are recalled in section \ref{sec:causality}). Integral curves of $C$ (i.e. Lipschitz curves $\gamma$ such that $\dot\gamma(t)\in C(\gamma(t))$ a.e.) are exactly future directed causal curves (in the topological sense). The data of $C$ is equivalent to the data of the conformal structure $[g]$ (two non degenerate quadratic forms of non definite signature share the same isotropic cone if and only if they differ by a multiplicative constant). \\
\indent Starting from a point $p\in M$, the set of endpoints of integral curves starting at $p$ is the \textbf{causal future} $J^+(p)$.\\
\indent If we compare this situation to the classical setting of a vector field on a manifold, we realise that we are missing something: special parametrisations of integral curves. Since we consider all causal curves, as opposed to only timelike curves, we cannot use Lorentzian arc length. For this purpose, we will always consider an auxiliary Riemannian metric $h$ on $M$, and consider the set $J^+_{t,T}(p)$ of endpoints  $\gamma(1)$ of future directed causal curves $\gamma :  \intff{0}{1} \rightarrow M$ such that $\gamma(0)=p$ and $t< \ell_h(\gamma)< T$, for some fixed $0<t<T$ and $p\in M$ (here $\ell_h$ denotes the Riemannian length with respect to $h$).

%and denote by $J^+_t(p)$ the subset of $M$ consisting of endpoints $\gamma(1)$ of future directed causal curves $\gamma :  \intff{0}{1} \rightarrow M$ such that $\gamma(0)=p$ and $\ell_h(\gamma)> t$.\\ 
%\indent When there is an ambiguity about the Riemannian metric, we will call this set $J^+_{t,h}(E)$.  Let us remark that this definition depends strongly on the auxiliary Riemmanian metric, and so will the definitions to come.  With a little more work and some slightly different definitions, some of these notions would not depend on such a choice (if we replace constants $T$ by positive continuous functions then we avoid the scale problem), but it is not necessary for the applications that we propose.  As we will see later, it will be important to choose a Riemannian metric with some special properties.\\
%\indent If $E$ is a subset of $M$, then we let $J^+_t(E)=\bigcup_{x\in E}J^+_t(x)$. If $0<t<T$, we will also consider the set $J^+_{t,T}(x)$ of endpoints  $\gamma(1)$ of future directed causal curves $\gamma :  \intff{0}{1} \rightarrow M$ such that $\gamma(0)=x$ and $t< \ell_h(\gamma)< T$.\\
The aim of sections \ref{sec:adapted_metrics} and \ref{sec:continuity_future} is to prove the following continuity result.

\begin{theorem} \label{theo:continuity_future}
Let $(M,g)$ be a spacetime. There is a complete Riemannian metric $h$ on $M$ such that the map $\overline{J^+_{t,T}}:M\to \mathrm{Comp}(M)$ is continuous, where $\mathrm{Comp}(M)$ is the set of compact subsets of $M$ endowed with the Hausdorff topology.
\end{theorem}
Note that the completeness of $h$ guarantees that $\overline{J^+_{t,T}(p)}$ is compact for all $p\in M$.

\subsection{Time function vs Lyapunov function} A time function in Lorentzian geometry is a function that increases along a future directed causal curve.  Hawking's Theorem states that  the existence of such a function is equivalent to stable causality.  A Lyapunov function in classical dynamical systems is  a function that is non increasing along orbits, and decreasing along certain orbits (physically it is seen as an energy function).  One of the main differences is that in dynamical systems, a flow that carries a function decreasing along all orbits is said to have poor dynamics, and isn't very interesting.  On the opposite, a Lorentzian manifold is physically relevant if it satisfies a certain number of causality conditions.  Since the goal of our work  is to use the techniques of construction of Lyapunov functions in order to construct time functions, our work will take place in a general setting with no particular causality conditions.  Therefore we cannot construct time functions right away, and we need to define some notion of  partial time function.

\begin{defi} Let $(M,g)$ be a spacetime (i.e.  a time oriented Lorentzian manifold), and $E$ a subset of $M$.  We call  $E$-time function   a continuous map $\tau :  M \rightarrow \R$ such that :  \begin{enumerate}
\item $\forall x\in M~\forall y\in J^+(x)~~\tau(y) \ge \tau(x)$ \item $\forall x\in E~\forall y\in J^+(x) \setminus \{ x
\}~~\tau(y) > \tau(x)$ \end{enumerate} \end{defi}

With this terminology, a time function is a $M$-time function.\\
\indent In \cite{conley}, Charles Conley proved that the existence of Lyapunov functions is related to attractors and chain recurrence.  His work was done in the case of a flow on a compact metric space, and it was extended to separable metric spaces by Hurley (see \cite{hurley_92},\cite{hurley_95},\cite{hurley_98}, and \cite{choi} for some corrections).  Since compact Lorentzian manifolds cannot carry time functions (see Proposition \ref{compact_not_causal}), we will not restrict ourselves to the compact case.  This will lead to some technical difficulties similar to the ones that were tackled in Hurley's work.
%\indent It is noteworthy that Hawking's method for constructing time functions could not provide $E$-time functions, as some causality is important in order to prove the continuity  of the function  in \cite{hawking}.

\newpage

\subsection{Chain recurrence and attractors in spacetimes}

\subsubsection{The chain recurrent set}
Closed causal curves are an obvious obstruction to the existence of a time function: a time function would have to be constant along such a curve. However, the absence of such curves is not enough to guarantee the existence of a time function, and Hawking stated that the right condition is stable causality, i.e. the absence of closed causal curves for nearby metrics. Chain recurrence is a different approach to this problem: instead of looking at causal curves for nearby metrics, we try to figure out if there are almost closed curves for a given metric.\\
\indent We follow \cite{hurley_92} and define $\PP(M)$ as the set of continuous functions from $M$ to $\intoo{0}{+\infty}$.

\begin{defi} \label{def_chain_rec} Let $\e \in \PP(M)$, $T>0$ and $p,q\in M$.\\ An $(\e,T)$-chain from $p$ to $q$ is a finite sequence of future directed causal curves $(\gamma_i:[0,1]\to M)_{i=1,\dots ,k}$ of length at least $T$ such that: 
\begin{enumerate} \item $d(p,\gamma_1(0))\le \e( p)$
\item $d(\gamma_i(1),\gamma_{i+1}(0))\le \e(\gamma_i(1))$ for all $i<k$
\item $d(\gamma_k(1),q)\le \e(q)$
\end{enumerate}

A point $p\in M$ is said to be chain recurrent if for any $\e \in \PP(M)$ and $T>0$ there is an $(\e,T)$-chain from $p$ to $p$. We will denote by $R(g)$ the set of chain recurrent points.
\end{defi}

Note that this definition involves the choice of an auxiliary Riemannian metric in order to define lengths of curves and distances between points.

\subsubsection{Attractors and partial time functions}

\indent At the heart of Conley's proof of the existence of Lyapunov functions lies the notion of attractors. Sets that attract orbits will be natural candidates to be the place where a Lyapunov functions reaches its minimum, and the Lyapunov function can be thought of as a distance to an attractor.\\
Given an auxiliary Riemannian metric $h$, a point $p\in M$ and $t>0$, we denote by $J^+_t(p)$ the subset of $M$ consisting of endpoints $\gamma(1)$ of future directed causal curves $\gamma :  \intff{0}{1} \rightarrow M$ such that $\gamma(0)=p$ and $\ell_h(\gamma)> t$.\\ 
\indent When there is an ambiguity about the Riemannian metric, we will call this set $J^+_{t,h}(E)$.  Let us remark that this definition depends strongly on the auxiliary Riemmanian metric, and so will the definitions to come.  With a little more work and some slightly different definitions, some of these notions would not depend on such a choice (replacing constants $t$ by positive continuous functions could be a way of avoiding the scale problem), but it is not necessary for the applications that we propose.  As we will see later, it will be important to choose a Riemannian metric with some special properties.\\
\indent If $E$ is a subset of $M$, then we let $J^+_t(E)=\bigcup_{p\in E}J^+_t(p)$.\\

\begin{defi} \label{defi_attractor}  An open set $U\subset M$ is said to be a pre attractor if there is  $t_0>0$ such that $\overline{J^+_{t_0}(U)} \subset U$.\\ The set $A=\bigcap_{t\ge t_0}\overline{J^+_t(U)}$ is called the attractor.\\ The set $B(A,U)=\bigcup_{t\ge 0} \{p\in M \vert \overline{J^+_t(p)}\subset U\}$ is called the basin of $U$-attraction.\\ If $\mathcal U$ is the set of pre attractors sharing the same attractor $A$, then the basin of attraction is  $B(A)=\bigcup_{U\in\mathcal U}B(A,U)$  \end{defi}

%\indent Note that the distinction between $B(A,U)$ and $B(A)$ has to be made: in the Minkowski space $\R^{1,n-1}$, the empty set is an attractors, with pre attractor $I^+(x)$ for any point $x\in \R^{1,n-1}$. In this case, we see that $B(\emptyset, I^+(x))=I^+(x)$, but $B(\emptyset)=\R^{1,n-1}$.\\
\indent Note that the attractor $A$ may be empty. The central result of this paper is the following theorem, which is the equivalent in Lorentzian
geometry of Conley's Theorem for flows.

\begin{theo} \label{mainthm} If $A$ is  an attractor in a spacetime $(M,g)$ (for an adapted Riemannian metric), then there is a $B(A)\setminus A$-time function.  \end{theo}

Adapted Riemannian metrics will be defined and discussed in section \ref{sec:adapted_metrics}.\\
\indent When looking for chain recurrent points, the best places to look into are attractors and repellers (i.e. attractors for the same metric with reversed time orientation).

\begin{prop} \label{chains_attractors} Let $(M,g)$ be a spacetime, and let $p\in M$. If $p\notin R(g)$, then there is an attractor $A$ such that $p\in B(A)\setminus A$.
\end{prop}

A similar result can be obtained in stably causal spacetimes.

\begin{prop} \label{stable_attractors} Let $(M,g)$ be a stably causal spacetime. Given any point $p\in M$, there is an attractor $A$ such that $p\in B(A)\setminus A$.
\end{prop}

The statement of Theorem \ref{stable_no_chains} is actually misleading: our strategy for constructing time functions relies on Theorem \ref{mainthm} and Proposition \ref{stable_attractors}, it is the existence of a time function that implies Theorem \ref{stable_no_chains}.

\subsection{Overview}
We will start by reviewing some classical definitions and results of causality theory. In section \ref{sec:adapted_metrics}, we will construct Riemannian metrics that are "adapted" to a spacetime, in the sense that they have nice properties when considering lengths of limit curves. Section \ref{sec:continuity_future} is devoted to the continuity of the future map (Theorem \ref{theo:continuity_future}).\\
Attractors will be studied in section \ref{sec:attractors}, where Theorem \ref{mainthm} will be proved. Finally, we will study the link between attractors and chain recurrence in section \ref{sec:chain_recurrence}, and give applications of attractors to stably causal spacetimes in \ref{sec:stably_causal}.

\section{Causality and time functions}  \label{sec:causality}
We recall the basic definitions of Lorentzian geometry, for further information see \cite{BEE}, \cite{O'N} or the introduction of \cite{o'neill_trous_noirs}. For a survey on causality theory, see \cite{c_causality}.\\
\indent A Lorentzian metric $g$ on a manifold $M$ is a symmetric $(2,0)$-tensor of signature $(-,+,\dots,+)$. A tangent vector $v\in T_pM$ is called \textbf{timelike} if $g_p(v,v)<0$, and \textbf{causal} if  $g_p(v,v)\leq 0$.\\
\indent A Lorentzian manifold is time-orientable if there is a continuous everywhere timelike vector field. A time orientation on a time orientable Lorentzian manifold is an equivalence class of everywhere timelike vector fields, where two such vector fields $T,T'$ are considered equivalent if $g(T,T')<0$ everywhere (i.e. they lie in the same connected component of $g_p^{-1}(\intoo{-\infty}{0})$ at every point $p\in M$).\\
\indent A causal vector $v\in TM$ is \textbf{future directed} if $g(v,T)\leq 0$ where $T$ is any vector field in the time orientation. A  future directed curve is a Lipschitz curve whose tangent vector is causal and future directed almost everywhere. A past directed curve is a curve such that the same curve reparametrized with reversed orientation is future directed.\\
 \indent By a \textbf{spacetime}, we always refer to a time oriented Lorentzian manifold $(M,g)$.
\subsection{A glimpse at the causal ladder} If $(M,g)$ is a spacetime,  the \textbf{chronological future} $I^+(p)$ of a point $p\in M$ is the set of endpoints of future directed timelike curves starting at $p$. Its \textbf{chronological past} $I^-(p)$ is the set of endpoints of past directed timelike curves starting at $p$.\\
\indent The \textbf{causal future} $J^+(p)$ (resp. \textbf{causal past} $J^-(p)$) is the set of endpoints of future directed (resp. past directed) causal curves starting at $p$. \\
\indent Let us state a few basic properties for the chronological and causal futures (see \cite{penrose} for a proof).
\begin{prop} If $q\in J^+(p)$, then $J^+(q)\subset J^+(p)$ and $I^+(q)\subset I^+(p)$.\\
The chronological future $I^+(p)$ is open. \end{prop}
The causal future is not necessarily closed, but we always have the chain of inclusions $I^+(p)\subset J^+(p)\subset \overline{I^+(p)}$.\\
\indent A spacetime $(M,g)$ is said to be \textbf{chronological} if there is no closed timelike curve, i.e. if  $p\notin I^+(p)$ for all $p\in M$. We say that $(M,g)$ is \textbf{causal} if there is no closed causal curve. The first implication of the chronological character of a spacetime concerns its  topology.

\begin{prop}\label{compact_not_causal} If $(M,g)$ is chronological, then $M$ is not compact. \end{prop}
\begin{proof} Consider a time orientation $T$ on $M$, i.e. a timelike vector field, and let $\p^t$ be its flow. If $M$ is compact, then the flow $\p^t$ has recurrent points (a point $p$ is recurrent if there is a sequence $t_n\to +\infty$ such that $\p^{t_n}(p)\to p$). Let $p\in M$ be recurrent. Since $I^-(\p^1(p))$ is open and contains $p$, it contains $\p^{t}(p)$ for some $t>1$. This shows that $\p^1(p)\in I^+(\p^{t}(p)) \subset I^+(\p^1(p))$, therefore $(M,g)$ is not chronological.
\end{proof}

Proposition \ref{compact_not_causal} is the source of  technical difficulties that we will face.  The "standard" proof of this classic result can be found in \cite{penrose}.\\
\indent Note that another proof using dynamical systems is possible: on a compact manifold, any vector field is arbitrarily close (in the compact-open topology) to a vector field that has a closed orbit (this is a consequence of the existence of recurrent points and of the $ C^0$-Closing Lemma). Since the set of timelike vector fields is open in the compact-open topology, there is always a timelike vector field with a closed orbit, which is a closed timelike curve. \\

We say that $(M,g)$ is \textbf{strongly causal} if every $p\in M$ has arbitrarily small  neighbourhoods $U$ such that the intersection of $U$ with a causal curve is always connected. Such a neighbourhood $U$ is called \textbf{causally convex}. Since small neighbourhoods do not contain closed causal curves (one can find charts such that the first coordinate is increasing along future directed causal curves), a strongly causal spacetime is causal.\\

\indent Let us define the partial order $\prec$ on the set of Lorentz metrics on $M$ by $g\prec g'$ if every non zero causal vector for $g$ is timelike for $g'$ (i.e. if the light cone of $g'$ is larger that the light cone of $g$). We say that $(M,g)$ is \textbf{stably causal} if there is $g\prec g'$ such that $(M,g')$ is causal.

\begin{prop} \label{stable_implies_strong} A stably causal spacetime is strongly causal. \end{prop}
\begin{proof} Let $(M,g)$ be a stably causal spacetime, and let $p\in M$. Let $g'\succ g$ be a causal metric. Let $U$ be a chart around $p$  with coordinates $(t,x_1,\dots,x_{n-1})$ such that $g_p =-dt^2+dx_1^2+\cdots +dx_{n-1}^2$ (such coordinates can be obtained with the exponential map). For $\alpha >0$, we denote by $g_\alpha$ the metric $-\alpha dt^2+dx_1^2+\cdots dx_{n-1}^2$ on $U$. If $\alpha'>\alpha>1$ and $\alpha'$ is sufficiently close to $1$, then $g\prec g_\alpha\prec g_{\alpha'}\prec g'$ on $U$.\\
\indent For $q\in U$ and $h$ a Lorentz metric on $U$, we denote by $I^{\pm}_{U,h}(q)$ the chronological past and future of $q$ in the spacetime $(U,h)$.\\
\indent Assume that $p=(0,\dots,0)$ in coordinates, and let  $p_+=(\delta,0,\dots,0)$ for $\delta>0$. Let $\e>0$ be small enough so that every future (resp. past) directed causal curve starting at $p$ and escaping $U$ meets the level $\{t=\e\}$ (resp. $\{t=-\e\}$).\\
\indent  If $\delta>0$ is small enough so that $(\delta+\e)^2<\frac{\alpha'}{\alpha} \e^2$, then we obtain:
$$I^-_{U,g_\alpha}(p_+) \cap \{ t=-\e\} \subset I^-_{U,g_{\alpha'}}(p)\cap \{t=-\e\}$$
\indent Since $g\prec g_\alpha\prec g_{\alpha'}\prec g'$, this implies that:
$$ I^-_{U,g}(p_+) \cap \{ t=-\e\} \subset I^-_{U,g'}(p)\cap \{t=-\e\}$$
\indent If $\gamma$ is a future directed causal curve (for $g$), and $\gamma(1)=p_+$, then there is $t<1$ such that $\gamma(t)\in I^-_{g'}(p)$ (see Figure \ref{fig:stable_implies_strong}).\\
\indent Similarly, if $\delta$ is small enough, then the point $p_-=(-\delta,0,\dots,0)$ satisfies:
$$ I^+_{U,g}(p_-) \cap \{ t=\e\} \subset I^+_{U,g}(p)\cap \{t=\e\}$$
\indent If $\gamma$ is a future directed causal curve (for $g$), and $\gamma(0)=p_-$, then there is $t>0$ such that $\gamma(t)\in I^+_{g'}(p)$.\\
\indent Now let $W=I^-_{U,g}(p_+)\cap I^+_{U,g}(p_-)$. If $\gamma$ were a closed causal curve whose intersection with $W$ is disconnected, then we can assume that $\gamma(0)=p_-$ and $\gamma(1)=p_+$. There are $t_1>0$ such that $\gamma(t_1)\in I^+_{g'}(p)$  and $t_1<t_2<1$ such that $\gamma(t_2)\in I^-_{g'}(p)$. This implies that $p\in I^+_{g'}(p)$, which is absurd because  $(M,g')$ is causal. This shows that $W$ is causally convex, hence the strong causality of $(M,g)$.

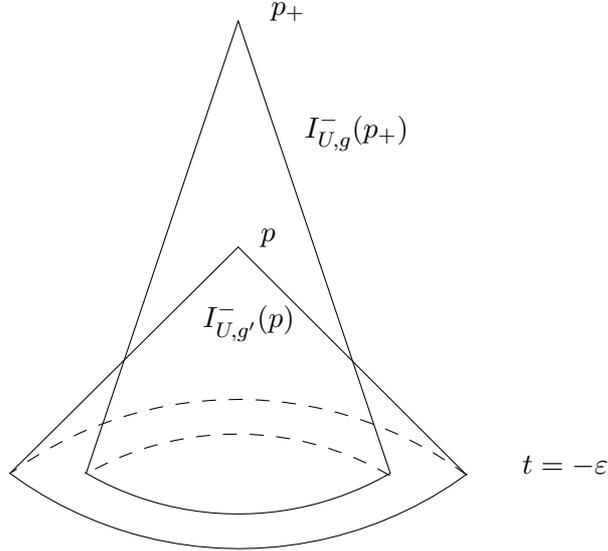
\begin{figure}[h]
\begin{tikzpicture}[line cap=round,line join=round,>=triangle 45,x=1.0cm,y=1.0cm]
\clip(-3.3,-2) rectangle (16.9,6.3);
\draw (0,-1)-- (3,2);
\draw (3,2)-- (6,-1.02);
\draw (5,-1)-- (3,5);
\draw (3,5)-- (1,-1);
\draw [shift={(3.01,3.04)}] plot[domain=4.07:5.35,variable=\t]({1*5.04*cos(\t r)+0*5.04*sin(\t r)},{0*5.04*cos(\t r)+1*5.04*sin(\t r)});
\draw [shift={(3,2.43)}] plot[domain=4.18:5.24,variable=\t]({1*3.97*cos(\t r)+0*3.97*sin(\t r)},{0*3.97*cos(\t r)+1*3.97*sin(\t r)});
\draw [shift={(2.99,-5.06)},dash pattern=on 5pt off 5pt]  plot[domain=0.93:2.2,variable=\t]({1*5.04*cos(\t r)+0*5.04*sin(\t r)},{0*5.04*cos(\t r)+1*5.04*sin(\t r)});
\draw [shift={(2.98,-4.45)},dash pattern=on 5pt off 5pt]  plot[domain=1.04:2.09,variable=\t]({1*3.97*cos(\t r)+0*3.97*sin(\t r)},{0*3.97*cos(\t r)+1*3.97*sin(\t r)});
\draw (3.16,2.38) node[anchor=north west] {$p$};
\draw (3.3,5.38) node[anchor=north west] {$p_+$};
\draw (6.6,-0.64) node[anchor=north west] {$t=-\varepsilon$};
\draw (3.74,3.92) node[anchor=north west] {$I^-_{U,g}(p_+)$};
\draw (2.4,1.4) node[anchor=north west] {$I^-_{U,g'}(p)$};
\end{tikzpicture}
\caption{Finding a causally convex neighbourhood} \label{fig:stable_implies_strong}
\end{figure}

\end{proof}

Note that this proof is also inspired by ideas of dynamical systems: the notion of strong causality is very similar to the absence of non wandering points for a flow, and the Closing Lemma states that a non wandering point is a periodic point for a nearby flow.\\
\indent Proposition \ref{stable_implies_strong} is a well known result, however it seems that every available proof in the literature (see e.g. \cite{minguzzi_sanchez}) uses the existence of a time function. We gave a direct proof because we will use the strong causality of stably causal spacetimes in section \ref{sec:stably_causal} in order to produce time functions, so we want to make sure that this construction does not rely on the existence of such functions.

\subsection{Time functions}
A \textbf{time function} is a continuous function $\tau:M\to \R$ such that $\tau \circ \gamma$ is increasing for any future directed causal curve $\gamma : I\subset \R \to M$. The existence of a time function implies that $(M,g)$ is causal. It is also easy to see that it implies strong causality (the sets $\tau^{-1}(\intoo{a}{b})$ are causally convex). A famous theorem of Hawking states that the right condition is stable causality. 

\begin{theo} \label{theo_hawking} A spacetime admits a time function if and only if it is stably causal. \end{theo}

Note that both implications are non trivial.  A \textbf{temporal function} is a smooth function $\tau: M\to \R$ whose gradient $\nabla \tau$ is timelike (the gradient $\nabla \tau$ is the unique vector field such that $d\tau_p(v)=g_p(\nabla \tau(p),v)$ for all $(p,v) \in TM$). A temporal function is a time function, but a smooth time function needs not be a temporal function. However, existence of one or the other is equivalent.

\begin{theo} \label{theo_BS} A spacetime admits a time function if and only if it admits a temporal function. \end{theo}

Theorem \ref{theo_BS} was finally proved by Bernal and Sanchez in a series of papers (\cite{BS_surface_lisse}, \cite{BS_fonction_lisse}) resolving several classical yet unsolved problems around temporal functions. Hawking proved that stable causality implies the existence of a time function, but he did not prove the converse (in \cite{hawking}, one can also find the statement of Theorem \ref{theo_BS}, however the proof quotes a paper which contains mistakes). In \cite{fathi_siconolfi}, Fathi and Siconolfi give a direct proof of the fact that a stably causal spacetime admits a temporal function (the proof is completely independent from any previous work, and is based on some tools of weak KAM theory, similar methods also apply to prove Conley's Theorem, see \cite{pageault}). Chru\'sciel, Grant and Minguzzi showed in  \cite{chrusciel}  that the time functions produced by Geroch and Hawking can actually be chosen to be smooth (see also \cite{minguzzi}).  For globally hyperbolic spacetimes (which we will not discuss in this paper), Suhr \cite{suhr} used Sullivan's theory of structure cycles \cite{sullivan} to produce specific temporal functions. So far, it seems that the only way to prove that a spacetime admitting a time function is stably causal uses Theorem \ref{theo_BS} (indeed, a temporal function is still a temporal function for nearby metrics, hence the stable causality). It would be interesting to find a direct proof of this fact (i.e. without using temporal functions). In section \ref{sec:chain_recurrence}, we will show that the existence of time function implies that the chain recurrent set is empty, however showing  that this implies stable causality demands technical results that we were not able to obtain.

%------------------------------------------------------------------------------------------------------------------------------
%------------------------------------------------------------------------------------------------------------------------------
%------------------------------------Elementary results-----------------------------------------------------------------
%------------------------------------------------------------------------------------------------------------------------------
%------------------------------------------------------------------------------------------------------------------------------

%------------------------------------------------------------------------------------------------------------------------------
%------------------------------------------------------------------------------------------------------------------------------
%------------------------------------Choosing-the-right-metric-----------------------------------------------------------------
%------------------------------------------------------------------------------------------------------------------------------
%------------------------------------------------------------------------------------------------------------------------------

\section{Adapted Riemannian metrics} \label{sec:adapted_metrics}
We will start by choosing a Riemannian metric that has a nice properties when studying lengths of limits of causal curves.

 \begin{defi} \label{adapted} Let $(M,g)$ be a spacetime.  We say that a Riemannian metric $h$ on $M$ is adapted to $g$ if :  \begin{itemize} \item For any point $p\in M$, there is a coordinate chart around $p$ and a constant $\lambda>0$ such that, in coordinates, $g$ at $p$ is $\lambda
(-dx_1^2+dx_2^2+\dots +dx_n^2)$ and $h$ at $p$ is $dx_1^2+dx_2^2+\dots +dx_n^2$ \item $h$ is complete \end{itemize}
\end{defi}

The standard example is the Euclidean metric on $\R^n$ that is adapted to the Minkowski metric.  It is quite easy to see that such a metric always exist.

\begin{prop} \label{adapted_exists} If $(M,g)$ is a spacetime,  an adapted Riemannian metric always exists.
\end{prop}

\begin{proof} Let $T$ be an everywhere timelike and future directed vector field.  For $p\in M$, we will denote by $E(p)$ the orthogonal space for $g$ of $T(p)$ (hence $T_pM = Vect(T(p)) \oplus E(p)$).  Let $h(g)$ be the Riemannian metric defined by:  \begin{displaymath} h(g)_p( x T(p) + \underbrace{u}_{\in E(p)}, y T(p) + \underbrace{v}_{\in E(p)})
= -xy \underbrace{g_p(T(p),T(p))}_{<0} + \underbrace{g_p(u,v)}_{\ge 0} \end{displaymath} The idea of this construction (called a Wick rotation) is that with the Minkowski metric $-dx_1^2+dx_2^2+\dots+dx_n^2$ and the vector field $\frac{\partial}{\partial x_1}$, we find the Euclidean metric $dx_1^2+\dots+dx_n^2$.  Let us remark that if we multiply $g$ by a positive function, then we also multiply $h(g)$ by the same function.  Since every conformal class of Riemannian metrics contains a complete metric, we can find $\tilde{g}$ in the conformal class of $g$ such that $h(\tilde{g})$ is complete.  Since the notion of an adapted metric depends only on the conformal class, we can now work exclusively with $\tilde{g}$.  Consider $p\in M$, and $(e_1,e_2,\dots,e_n)$ an orthonormal frame of $T_pM$ for $\tilde{g}$ such that $e_1=\frac{T(p)}{\sqrt{-\tilde{g}_p(T(p),T(p))}}$.  It is also an orthonormal frame for $h(\tilde{g})$.  Applying the exponential map for $\tilde{g}$ shows that $h(\tilde{g})$ is adapted to $\tilde{g}$ (with constant $\lambda=1$), and therefore to $g$.

\end{proof}

We will now only consider adapted Riemannian metrics, unless specified.  We will be particularly interested in properties of such metrics regarding the length of limits of sequences of causal curves.  Since adaptiveness is a local property, our results will all start with a local version, and then will be extended globally.\\
\indent The length function is always upper semi continuous, i.e.  $\ell_h(\gamma) \le \liminf \ell_h(\gamma_k)$ for a sequence $(\gamma_k)$ converging to $\gamma$ in the compact open topology, but it is not lower semi continuous.  We will show that we have a weak version of lower semi continuity for sequences of causal curves.\\
\indent Since we are going to consider limits of sequences of causal curves, we are going to need to extend the notion of causal curve to some continuous curves. We say that a curve $\gamma:I\subset \R\to M$ is \textbf{future directed} if it is locally Lipschitz and its derivative is almost everywhere in the future cone of $g$. An important fact is that this definition does not change the future or past: all points that are reachable by a future directed curve (in the topological sense) are reachable by a smooth future curve.

\begin{prop} Let $(\gamma_k)_{k\in\N}\in C(I,M)^\N$ be a sequence of future directed curves defined on an interval $I\subset \R$. If $(\gamma_k)$ converges in the compact open topology to a curve $\gamma$, then $\gamma$ is future directed. \end{prop}

\begin{prop} Let $(\gamma_k)_{k\in\N}\in C(\R,M)^\N$ be a sequence of future directed curves. Up to  changes of parameters, there is a subsequence that converges towards a future directed curve $\gamma$ in the compact open topology.   \end{prop}

The proofs are in section 7 of \cite{Ba05}.

\begin{prop} \label{length_limit} Let $h$ be a Riemannian metric adapted to $g$, and let $(\gamma_k)_{k\in \N}$ be a sequence of future directed causal curves converging to $\gamma$ in the compact open topology, then $\ell_h(\gamma) \ge \frac{1}{\sqrt{2}} \limsup \ell_h(\gamma_k)$ \end{prop}

Before we prove this result, let us prove a local version.

\begin{lemme} \label{length_limit_local} Let $p\in M$ and $C\in \intoo{0}{\frac{1}{\sqrt{2}}}$.  There is a neighbourhood $U_p$ of
$p$ such that for all sequence $(\gamma_k)$ of future curves in $U$ converging to $\gamma$, we have $\ell_h(\gamma) \ge C
\limsup \ell_h(\gamma_k)$.  \end{lemme}

\begin{proof}
Let us consider $U$ a coordinate neighbourhood of $p$ given by the definition of an adapted metric, and  choose $a\in \intoo{0}{1}$ and $b>1$.  We will denote by $\tilde{h}$ the Euclidean metric on $U$.  If we reduce $U$ sufficiently, then we have $a\tilde{h}_q(u,u) \le h_q(u,u) \le b\tilde{h}_q(u,u)$ for all $q\in U$ and $u\in T_qM$ because of the continuity of both metrics and the equality at $p$.\\ 
\indent Let us also choose $\alpha >1$ and denote by $\tilde{g}$ the constant Lorentzian metric on $U$ given by $-\alpha dx_1^2 +dx_2^2 + \dots +dx_n^2$.  If we reduce $U$ sufficiently, then $g<\tilde{g}$ on $U$ (i.e.  a non zero causal vector for $g$ is timelike for $\tilde{g}$), since $g$ has the same light cone at $p$ as $-dx_1^2+dx_2^2+\dots+dx_n^2$.\\ 
\indent Let us now consider a sequence $(\gamma_k)$ of future curves in $U$ converging to $\gamma$. Since these curves are causal for $g$ and therefore for $\tilde{g}$, they can be parametrized by the first coordinate: $\gamma_k(t) = \gamma_k(0) + (t,x_2^k(t),\dots,x_n^k(t))$ and $\gamma(t) = \gamma(0) + (t,x_2(t),\dots,x_n(t))$.  Since these curves are causal for $\tilde{g}$, we have $\dot{x}_2^k(t)^2+\dots+\dot{x}_n^k(t)^2 \le \alpha$.  Let us denote by $x_k$ (resp.  $x$) the first coordinate of the endpoint of $\gamma_k$ (resp.  $\gamma$).  We have: 
 \begin{eqnarray*} \ell_{\tilde{h}}(\gamma_k) &=& \int_0^{x_k} \underbrace{\sqrt{1 + \dot{x}_2^k(t)^2 + \dots + \dot{x}_n^k(t)^2}}_{\le \sqrt{1+\alpha}} dt \\ &\le& x_k \sqrt{1+\alpha} \end{eqnarray*}
\indent We also have $\ell_{\tilde{h}}(\gamma) = \int_0^{x} \underbrace{\sqrt{1 + \dot{x}_2(t)^2 + \dots + \dot{x}_n(t)^2}}_{\ge
1} dt \ge x$, therefore:

\begin{eqnarray*} \ell_{\tilde{h}}(\gamma) &\ge& x \\ &\ge& \lim x_k \\ &\ge& \frac{1}{\sqrt{1+\alpha}} \limsup
\ell_{\tilde{h}}(\gamma_k) \end{eqnarray*}

By choosing $a$, $b$ and $\alpha$ such that $\sqrt{\frac{a}{b(1+\alpha)}} \ge C$, we obtain :  \begin{eqnarray*}
\ell_{h}(\gamma) &\ge& \sqrt{a} ~\ell_{\tilde{h}}(\gamma) \\ &\ge& \sqrt{\frac{a}{1+\alpha}} \limsup
\ell_{\tilde{h}}(\gamma_k) \\ &\ge& \sqrt{\frac{a}{b(1+\alpha)}} \limsup \ell_{h}(\gamma_k) \\ &\ge& C \limsup
\ell_{h}(\gamma_k) \end{eqnarray*}

\end{proof}

\begin{proof}[Proof of Proposition \ref{length_limit}.] Let us consider a sequence $(\gamma_k)$ of future causal curves
converging to $\gamma$.  For all $p\in M$, let us choose a neighbourhood given by Lemma \ref{length_limit_local}.  Let $\intff{a}{b}$ be a compact interval in the domain of these curves.  Since $\gamma(\intff{a}{b})$ is compact, we can consider a finite cover $\gamma(\intff{a}{b}) = \bigcup_{1\le i \le m} U_{\gamma(t_i)}$ with $t_1<\dots<t_m$.  Since $\gamma$ is continuous, we can
find numbers $a=s_1<s_2<\dots<s_{m+1}$ such that $\gamma(\intff{s_i}{s_{i+1}})\subset U_{\gamma(t_i)}$.  Let $\gamma^i$ (resp. $\gamma_k^i$) be the restriction of $\gamma$ (resp.  $\gamma_k$) to $\intff{s_i}{s_{i+1}}$.  \\ 
\indent Let $C\in \intoo{0}{\frac{1}{\sqrt{2}}}$. Given $i\in \{1,\dots,m\}$, then for $k$ sufficiently large, the curve $\gamma_k^i$ lies in $U_{\gamma(t_i)}$ and therefore $\ell_h(\gamma^i)\ge C \limsup \ell_h(\gamma_k^i)$.  After a sum over $i$, we obtain $\ell_h(\gamma)\ge C \limsup \ell_h(\gamma_k)$ and the proposition is proved by considering the limit when $C$ tends to $\frac{1}{\sqrt{2}}$.

\end{proof}

Even though this result will be useful in the next section, it will not be enough for our purpose, and we need to show that
we can choose another curve joining the same points that is longer.

\begin{prop} \label{long_curve} Let $(\gamma_k :  \intff{0}{1} \rightarrow M)$ be a sequence of future directed causal curves with constant speed, of length $\ell_k >0$ converging to $\ell>0$, and such that $\gamma_k$ converges uniformly to a curve $\gamma$.  Then for all $\e >0$, there exists a future directed causal curve $\eta_{\e}$ such that :  \begin{itemize} \item $\eta_{\e}(0) = \gamma(0)$ \item $\eta_{\e}(1) = \gamma(1)$ \item $\ell (1-\e) \le \ell_h(\eta_{\e}) \le \ell (1+ \e)$
\end{itemize} \end{prop}

Once again, let us start by formulating and proving a local version of this result.

\begin{lemme} \label{long_curve_local} Let $p\in M$ and let $\e >0$.  There is a closed neighbourhood $V$ of $p$ such that for any sequence $(\gamma_k :  \intff{0}{1} \rightarrow V)$ of future directed causal curves with constant speed, of length $\ell_k >0$ converging to $\ell>0$, and such that $\gamma_k$ converges uniformly to a curve $\gamma$, there exists a future
directed causal curve $\eta$ such that :  \begin{itemize} \item $\eta(0) = \gamma(0)$ \item $\eta(1) = \gamma(1)$ \item $\ell
(1-\e) \le \ell_h(\eta) \le \ell (1+ \e)$ \end{itemize} \end{lemme}

\begin{proof} Let us consider a coordinate neighbourhood $U$ of $p$ given by Definition \ref{adapted}.  For $\alpha \in
\intff{0}{1}$, let us denote by $g_{\alpha}$ the constant Lorentzian metric $-(1+\alpha)dx_1^2+dx_2^2+\dots +dx_n^2$, by
$g_{-\alpha}$ the metric $-(1-\alpha)dx_1^2+dx_2^2+\dots +dx_n^2$, by $\xi$ the Euclidean metric on $U$, by $\xi_{\alpha}$
(resp.  $\xi_{-\alpha}$ the Riemannian metric $(1+\alpha)\xi$ (resp.  $(1-\alpha)\xi$) on U.  We will choose $\alpha$ small
enough so that it will satisfy the following inequalities: 
\begin{enumerate} \item $(1+\alpha)^2\le 1+\e$
\item $\frac{1-\alpha}{1+\alpha} \frac{\sqrt{2-\alpha}}{\sqrt{2+\alpha}}  \ge 1-\e$
\item $\frac{1-\alpha}{1+\alpha} \frac{\sqrt{2-\alpha} - \sqrt{1-\alpha} + \sqrt{1-\alpha}}{\sqrt{2-\alpha}} \ge 1-\e$
\end{enumerate}

Let $V$ be a closed ball centred at $p$ for the infinite norm in coordinates, small enough so that $g_{-\alpha} < g <
g_{\alpha}$ and $\xi_{-\alpha} \le h \le \xi_{\alpha}$ on $V$.\\
\indent   We will denote by $f$ the first coordinate function on
$V$.\\

\textbf{First step: } Estimation of $\lambda=f(\gamma(1))-f(\gamma(0))$ \\

\indent Let $c:\intff{0}{1}\to V$ be a future directed causal curve.  It is also causal for $g_{\alpha}$.  If $c(t)=(c_1(t),\dots,c_n(t))$ in coordinates, then $-(1+\alpha) \dot{c}_1^2 + \dot{c}_2^2+\dots + \dot{c}_n^2 \le 0$ and
$\dot{c}_1 \ge 0$. We have:

\begin{eqnarray*} \ell_h(c) &\le& (1+\alpha)\ell_{\xi}(c) \\ &\le& (1+\alpha) \int_0^1 \sqrt{\dot{c}_1^2+\dots + \dot{c}_n^2}
 dt \\ &\le& (1+\alpha) \int_0^1 \sqrt{(2+\alpha)\dot{c}_1^2}dt\\ &\le& (1+\alpha)\sqrt{2+\alpha}(f(c(1))-f(c(0)))
 \end{eqnarray*}

\indent We also have:

\begin{eqnarray*} \ell_h(c) &\ge& (1-\alpha)\ell_{\xi}(c) \\ &\ge& (1-\alpha) \int_0^1 \sqrt{\dot{c}_1^2+\dots + \dot{c}_n^2}
 dt \\ &\ge& (1-\alpha) \int_0^1 \sqrt{\dot{c}_1^2}dt\\ &\ge& (1-\alpha)(f(c(1))-f(c(0)))
 \end{eqnarray*}

\indent By combining these inequalities and applying to $\gamma_k$, we obtain: $$\frac{1}{(1+\alpha)\sqrt{2+\alpha}}\ell_k \le f(\gamma_k(1))-f(\gamma_k(0)) \le
\frac{1}{1-\alpha} \ell_k$$ \indent The continuity of $f$  gives us:

\begin{displaymath} \frac{\ell}{(1+\alpha)\sqrt{2+\alpha}} \le \lambda \le \frac{\ell}{1-\alpha}
\end{displaymath}

\textbf{Second step: } First case: if $\gamma(1) \in J^+_{V,g_{-\alpha}}(\gamma(0))$.\\

\indent In this case, we construct $\eta$ as a piecewise causal geodesic for $g_{-\alpha}$.  We will consider the intersection $S$ of a horizontal hyperplane (in coordinates) located between $\gamma(0)$ and $\gamma(1)$ that meets the  intersection of the light cones for $J^+_{V,g_{-\alpha}}$ of $\gamma(0)$ and $\gamma(1)$, and of the half cone $J^+_{g_{-\alpha}}(\gamma(0))$.  For a point $p\in S$, we consider the curve $\eta_p$ obtained as the concatenation of the straight lines joining $\gamma(0)$ to $p$ and $p$ to $\gamma(1)$ (see Figure \ref{fig:construction_curve}).  The curve $\eta_p$ is causal for
$g_{-\alpha}$, and therefore causal (actually timelike) for $g$.  The maximum Euclidean length of $\eta_p$ is obtained when $p$ lies on the border of the cone (i.e.  $\eta_p$ is a null curve for $g_{-\alpha}$).  In this case, we have $\ell_{\xi}(\eta_p)=\sqrt{2-\alpha} \lambda \ge \frac{1}{1+\alpha} \frac{\sqrt{2-\alpha}}{\sqrt{2+\alpha}} \ell$ (this inequality is given by the first step), hence $\ell_h(\eta_p)\ge \frac{1-\alpha}{1+\alpha} \frac{\sqrt{2-\alpha}}{\sqrt{2+\alpha}} \ell$.  The minimum Euclidean length is obtained when $\eta_p$ is a straight line, in which case we have $\ell_{\xi}(\eta_p)=d_{\xi}(\gamma(0),\gamma(1))$, hence $\ell_h(\eta_p)\le (1+\alpha)d_{\xi}(\gamma(0),\gamma(1))$.  By using the integral expression of $\ell_h(\eta_p)$, one can see that it is a continuous function of $p$.  Therefore, the values of this map contains the interval
$J=\intff{(1+\alpha)d_{\xi}(\gamma(0),\gamma(1))}{\frac{1-\alpha}{1+\alpha} \frac{\sqrt{2-\alpha}}{\sqrt{2+\alpha}} \ell}$.  In order to conclude, we wish to see that $J\cap \intff{\ell(1-\e)}{\ell(1+\e)}\ne \emptyset$ (after that, we choose $p$ such that $\eta_p$ has length between $\ell(1-\e)$ and $\ell(1+\e)$ and we set $\eta=\eta_p$).  By using $\gamma(t)=\lim \gamma_k(t)$, we obtain $d_{\xi}(\gamma(0),\gamma(1))\le \frac{1}{1-\alpha}\ell \le (1+\alpha)\ell$, therefore $(1+\alpha)d_{\xi}(\gamma(0),\gamma(1))\le (1+\alpha)^2 \ell \le (1+\e)\ell$ (this is the first required inequality). We also chose $\alpha$ such that $\frac{1-\alpha}{1+\alpha} \frac{\sqrt{2-\alpha}}{\sqrt{2+\alpha}} \ell \ge
(1-\e)\ell$, which concludes to prove that $J\cap \intff{\ell(1-\e)}{\ell(1+\e)}\ne \emptyset$.\\

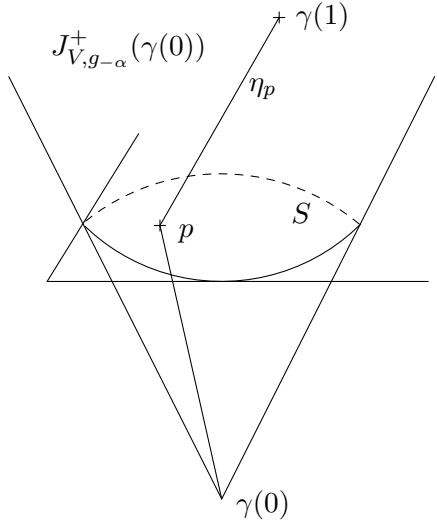
\begin{figure}[h]
\begin{tikzpicture}[line cap=round,line join=round,>=triangle 45,x=1.0cm,y=1.0cm,scale=1.4]
\clip(-2.7,-0.6) rectangle (16.9,4.7);
\draw (0,4)-- (2,0);
\draw (4,4)-- (2,0);
\draw (2.56,2.9) node[anchor=north west] {$S$};
\draw (2.6,4.8) node[anchor=north west] {$\gamma(1)$};
\draw (2.04,0.18) node[anchor=north west] {$\gamma(0)$};
\draw (0.3,4.54) node[anchor=north west] {$J^+_{V,g_{-\alpha}}(\gamma(0))$};
\draw (1.5,2.7) node[anchor=north west] {$p$};
\draw (1.42,2.59)-- (2.54,4.56);
\draw (1.42,2.59)-- (2,0);
\draw (2.16,4.08) node[anchor=north west] {$\eta_p$};
\draw (0.36,2.06)-- (3.94,2.06);
\draw (0.36,2.06)-- (1.22,3.46);
\draw [shift={(2,3.9)}] plot[domain=3.94:5.5,variable=\t]({1*1.84*cos(\t r)+0*1.84*sin(\t r)},{0*1.84*cos(\t r)+1*1.84*sin(\t r)});
\draw [shift={(1.99,1.06)},dash pattern=on 3pt off 3pt]  plot[domain=0.87:2.27,variable=\t]({1*2.02*cos(\t r)+0*2.02*sin(\t r)},{0*2.02*cos(\t r)+1*2.02*sin(\t r)});
\begin{scriptsize}
\draw [color=black] (2.54,4.56)-- ++(-1.5pt,0 pt) -- ++(3.0pt,0 pt) ++(-1.5pt,-1.5pt) -- ++(0 pt,3.0pt);
\draw [color=black] (1.42,2.59)-- ++(-1.5pt,0 pt) -- ++(3.0pt,0 pt) ++(-1.5pt,-1.5pt) -- ++(0 pt,3.0pt);
\end{scriptsize}
\end{tikzpicture}
\caption{Construction of the curve $\eta$} \label{fig:construction_curve}
\end{figure}

%Figure ancienne version
\begin{comment}
\begin{figure}[h]
\definecolor{uuuuuu}{rgb}{0.27,0.27,0.27}
\begin{tikzpicture}[line cap=round,line join=round,>=triangle 45,x=1.0cm,y=1.0cm,scale=1.4]
\clip(-2.7,-0.6) rectangle (16.9,4.7);
\draw (0,4)-- (2,0);
\draw (4,4)-- (2,0);
\draw (0.5,2)-- (3.5,2);
\draw (0.5,2)-- (1.2,3);

\draw (2.42,2.6) node[anchor=north west] {$S$};
\draw (2.6,4.8) node[anchor=north west] {$\gamma(1)$};
\draw (2.04,0.18) node[anchor=north west] {$\gamma(0)$};
\draw (0.46,4.4) node[anchor=north west] {$J^+_{V,g_{-\alpha}}(\gamma(0))$};
\draw [shift={(2.03,2.84)}] plot[domain=4.15:5.2,variable=\t]({1*1.52*cos(\t r)+0*1.52*sin(\t r)},{0*1.52*cos(\t r)+1*1.52*sin(\t r)});
\draw [shift={(1.93,-0.03)},dash pattern=on 3pt off 3pt]  plot[domain=1.08:1.99,variable=\t]({1*1.74*cos(\t r)+0*1.74*sin(\t r)},{0*1.74*cos(\t r)+1*1.74*sin(\t r)});
\draw (1.06,2.6) node[anchor=north west] {$p$};
\draw (1.42,2.59)-- (2.54,4.56);
\draw (1.42,2.59)-- (2,0);
\draw (2.02,3.58) node[anchor=north west] {$\eta_p$};
\begin{scriptsize}
\draw [color=black] (2.54,4.56)-- ++(-1.5pt,0 pt) -- ++(3.0pt,0 pt) ++(-1.5pt,-1.5pt) -- ++(0 pt,3.0pt);
\draw [color=black] (1.42,2.59)-- ++(-1.5pt,0 pt) -- ++(3.0pt,0 pt) ++(-1.5pt,-1.5pt) -- ++(0 pt,3.0pt);
\fill [color=uuuuuu] (2,0) circle (1.5pt);
%\draw[color=uuuuuu] (2.14,0.26) node {$L$};
\end{scriptsize}
\end{tikzpicture}
\caption{Construction of the curve $\eta$} \label{fig:construction_curve}
\end{figure}
\end{comment}

\textbf{Third step: } Second case: if $\gamma(1) \notin J^+_{V,g_{-\alpha}}(\gamma(0))$.\\

\indent In this case, we will simply show that the curve $\gamma$ already satisfies the desired properties.  We have $\ell_h(\gamma) \le \ell$ by upper semi continuity of the Riemannian length, which gives us the desired upper bound on $\ell_h(\gamma)$.  We also have $\ell_h(\gamma) \ge (1-\alpha) \ell_{\xi}(\gamma) \ge (1-\alpha) d_{\xi}(\gamma(0),\gamma(1))$ where $d_{\xi}$ is
the Euclidean distance in coordinates.  The rest of the proof is Euclidean geometry in coordinates.  Let us consider the vertical plane $P$ containing $\gamma(0)$ and $\gamma(1)$ and the horizontal hyperplane $S$ containing $\gamma(1)$.  Let $x$ (resp.  $y$) be intersection of $P$ and $\partial J^+_{V,g_{-\alpha}}(\gamma(0))$ (resp.  $\partial J^+_{V,g_{\alpha}}(\gamma(0))$) that is closest to $\gamma(1)$.  Let $z$ be the intersection of $S$ and the vertical line passing through $\gamma(0)$.  Since $\gamma(1) \notin J^+_{V,g_{-\alpha}}(\gamma(0))$ and $\gamma(1) \in J^+_{V,g_{\alpha}}(\gamma(0))$, we have $d_{\xi}(x,\gamma(1))\le d_{\xi}(x,y)$.\\ \indent Since $x$, $y$ and $z$ are on the same line, we have $d_{\xi}(x,y)=d_{\xi}(z,y)-d_{\xi}(z,x)$.  By using the Pythagorean Theorem, we obtain $d_{\xi}(z,x)=\sqrt{1-\alpha} \lambda$ and $d_{\xi}(z,y)=\sqrt{1-\alpha} \lambda$. We now have:

\begin{eqnarray*} d(\gamma(0),\gamma(1)) &\ge& d(\gamma(0),x) - d(x,\gamma(1)) \\
&\ge& \sqrt{2-\alpha} \lambda - d(x,y) \\
&\ge& (\sqrt{2-\alpha} - \sqrt{1-\alpha} + \sqrt{1-\alpha})\lambda \\
\end{eqnarray*}
\indent The first step and the third required inequality on $\alpha$ give us:

\begin{eqnarray*} \ell_h(\gamma) &\ge& (1-\alpha)(\sqrt{2-\alpha} - \sqrt{1-\alpha} + \sqrt{1-\alpha})\lambda \\
&\ge& \frac{1-\alpha}{1+\alpha} \frac{\sqrt{2-\alpha} - \sqrt{1-\alpha} + \sqrt{1-\alpha}}{\sqrt{2-\alpha}} \ell \\
&\ge& \ell(1-\e)
\end{eqnarray*}
\indent This gives the desired lower bound on $\ell_h(\gamma)$, which concludes the proof.

\end{proof}

\begin{proof}[Proof of Proposition \ref{long_curve}.]  Let $(\gamma_k)$ be a sequence of future directed curves converging to $\gamma$, and such that $\ell_h(\gamma_k) \rightarrow \ell$.  Let $\e >0$.  For $t\in \intff{0}{1}$, let us consider a neighbourhood $U_t$ given by Lemma \ref{long_curve_local}.  The open covering $\gamma(\intff{0}{1})\subset \bigcup_{t\in \intff{0}{1}} U_t$ admits a finite sub cover $\gamma(\intff{0}{1})\subset \bigcup_{1\le i \le m} U_{t_i}$ with $t_1 < t_2 < \dots < t_m$.  Let $0=s_0 < s_2 < \dots s_m=1$ such that $\gamma(\intff{s_i}{s_{i+1}})\subset U_{t_i}$ for all $i$.  Let us denote by $\gamma^i$ (resp. $\gamma_k^i$) the restriction of $\gamma$ (resp.  $\gamma_k$) to $\intff{s_i}{s_{i+1}}$.  For all $i$, there is a curve $\eta^i$ joining $\gamma(s_i)$ and $\gamma(s_{i+1})$ such that $ (s_{i+1} - s_i) \ell (1-\e) \le \ell_h(\eta^i) \le (s_{i+1} - s_i)
\ell (1+\e)$.  The concatenation $\eta$ of the curves $\eta^i$ joins $\gamma(0)$ and $\gamma(1)$ and satisfies $\ell (1-\e) \le \ell_h(\eta) \le \ell (1+\e)$.  \end{proof}

We will also use the following result that shows that, locally, causal curves cannot be arbitrarily long.

\begin{prop} \label{short_curves} Let $h$ be an adapted Riemannian metric to $g$ and let $x\in M$. For all $\e>0$, there is a neighborhood $U$ of $x$ such that for all causal curve $\gamma$ included in $U$, we have $\ell_h(\gamma)\le \e$. \end{prop}

\begin{proof} Once again, the idea is that it is simple in the Minkowski case, with an explicit neighborhood. Let $U$ be a coordinate neighborhood of $x$ given by the definition of an adapted metric. By reducing $U$, we can assume that $g<g_1$ and $h\le \xi_1$ on $U$ (we use the same notations as in the previous proposition), and that $-\frac{\e}{4\sqrt 3}\le x_1 \le \frac{\e}{4\sqrt 3}$ for all $(x_1,\dots,x_n)\in U$. Let $\gamma$ be a future directed causal curve in $U$, and let us write $\gamma(t)=(x_1(t),\dots,x_n(t))$. Since $\gamma$ is timelike for $g_1$, we can consider that $x_1(t)=t$.

\begin{eqnarray*} \ell_h(\gamma) &\le& 2\ell_{\xi}(\gamma) \\ &\le& 2 \int_{t_0}^{t_1} \sqrt{1+\dot{x}_2^2(t)+\cdots + \dot{x}_n^2(t)}
 dt \\ &\le& (1-\alpha) \int_{t_0}^{t_1} \sqrt{1+2}dt\\ &\le& 2\sqrt 3 (t_1-t_0) \\ &\le& \e
 \end{eqnarray*}
 
\end{proof}

With a hypothesis on causality, this tells us that causal curves that stay in a compact set have bounded length.

\begin{coro} \label{bounded_length} If $h$ is an adapted metric to $g$ and if $(M,g)$ is strongly causal, then for any compact set $K\subset M$, there is a constant $\ell >0$ such that for all causal curve $\gamma$ included in $K$, we have $\ell_h(\gamma)\le \ell$.
\end{coro}

\begin{proof} For all $x\in K$ we consider a neighborhood $U_x$ of $x$ given by Proposition \ref{short_curves} with $\e=1$. Since $(M,g)$ is strongly causal, by reducing $U_x$ we can assume that $U_x$ is causally convex. From the open covering $K\subset \bigcup_{x\in K}U_x$ we can extract a finite cover $K\subset \bigcup_{i=1}^n U_{x_i}$ where $x_1,\dots,x_n \in K$. Let $\ell =n$. If $\gamma$ is a causal curve included in $K$, we can divide $\gamma$ in a finite number $k$ of curves $\gamma_i$ such each of these curves is included in one $U_{x_j}$. Since they are causally convex, we have $k\le n$, and their length is at most $1$ by definition of $U_{x_j}$, therefore $\ell_h(\gamma) = \sum_{i=1}^k \ell_h(\gamma_i) \le k \le \ell$.

\end{proof}

%------------------------------------------------------------------------------------------------------------------------------
%------------------------------------------------------------------------------------------------------------------------------
%------------------------------------Hausdroff-continuity----------------------------------------------------------------------
%------------------------------------------------------------------------------------------------------------------------------
%------------------------------------------------------------------------------------------------------------------------------

\section{Hausdorff topology and continuity of $\overline{J^+_{t,T}}$}
\label{sec:continuity_future}
Let us recall that if $K_1$ and $K_2$ are two non empty compact subsets of a metric space $X$, then the Hausdorff distance between $K_1$ and $K_2$ is given by $d_H(K_1,K_2)=\inf \{ \e >0 \vert K_1\subset V_{\e}(K_2)$ and $K_2\subset V_{\e}(K_1)\}$, where $V_\e(K)$ denotes the $\e$-neighborhood of $K$.  It defines a metric on the space $\mathrm{Comp}(X)$ of non empty compact subsets of $X$.  The topology inherent to this metric, called the Hausdorff topology, does not depend on the choice of a metric on $X$, as long as it defines
the same topology.

The following result will be important later on.

\begin{prop} Let $X$ be a proper metric space (i.e.  closed balls are compact).  Let us consider a continuous map $f:Y\rightarrow \R$.  Then the map $g:\mathrm{Comp}(X) \rightarrow \R$ defined by $g(K)=\max \{ f(x) \vert x\in K \}$ is continuous.
\end{prop}

\begin{proof} We will separate the proofs of upper and lower semi continuity.  Let $K_0\in \mathrm{Comp}(X)$ and $\e >0$.  Let $x_0\in K_0$ such that $f(x_0)=g(K_0)$.  Let $\delta >0$ such that $d(x,x_0)<\delta$ implies $f(x)\ge f(x_0) - \e$.  Let us denote by $\varphi$ the map $K\mapsto d(x_0,K)$.  It is easy to check that $\varphi$ is $1$-Lipschitz for the Hausdorff metric and therefore continuous.  Since $\varphi(K_0)=0$, let $W$ be a neighborhood of $K_0$ in $\mathrm{Comp}(X)$ such that $\varphi(K)< \delta$ for $K\in W$.  For $K\in W$, we have $d(x_0,K)< \delta$ and therefore $B(x_0,\delta)\cap K \ne \emptyset$.  Let $x\in B(x_0,\delta)\cap K$, we have $f(x) \ge f(x_0) - \e$ and $g(K)\ge f(x)$, hence $g(K)\ge g(K_0) - \e$.  This concludes the proof of lower semi continuity.\\
\indent Let us now prove by contradiction that $g$ is upper semi continuous.  If it is not, then we can find $K \in \mathrm{Comp}(X)$, $\e>0$ and a sequence $K_n$ in $\mathrm{Comp}(X)$ such that $\lim K_n = K$ and $g(K_n) \ge g(K) + \e$ for all $n$.  Let $C=\{ x\in X / d(x,K)\le 1\}$.  Since $X$ is a proper metric space, $C$ is compact because it is closed and $\mathrm{diam}(C) \le \mathrm{diam}(K)+2 < \infty$.  For $n$ large enough, we have $K_n\subset V_1(K) \subset C$.  Let $x_n\in K_n$ such that
$f(x_n)=g(K_n)$.  Since $x_n\in C$, up to the choice of a subsequence, we can assume that $x_n$ tends to $x\in C$.  For
any $\eta >0$, we have $x_n \in K_n \subset V_{\eta}(K)$ for $n$ sufficiently large, and therefore $d(x_n,K)\le \eta$.
This shows that $\lim d(x_n,K)=0$, therefore $x\in K$.  We now have $g(K)\ge f(x) = \lim f(x_n) = \lim g(K_n) \ge g(K) +
\e$ which is impossible.  Therefore $g$ is lower semi continuous.  \end{proof}

Since the composition of continuous functions is continuous, the following is now obvious.

\begin{coro} Let $X$ be a topological space and $Y$ a proper metric space.  Let us consider continuous maps $F:X
\rightarrow \mathrm{Comp}(Y)$ and $f:  Y\to \R$.  Then the map $g :  X \rightarrow \R$ given by $g(x) =\max \{ f(y) /
y\in F(x) \}$ is continuous.  \end{coro}

We now wish to prove that, in some sense, the map that associates to a point its future is continuous.  If we try to deal with the whole future, then we face a major problem:  it is not generally continuous, and its continuity is actually related to causality conditions (see \cite{minguzzi_sanchez}).  This is why we consider the map  $x\mapsto \overline{J^+_{t,T}(x)}$. To prove its continuity, we will need some results that rely on the fact that we consider Riemannian metrics adapted to $g$.

\begin{lemme} Consider $0<t<T$, a sequence $(x_k)_{k\in \N} \in M^{\N}$ converging to $x\in M$, and a sequence
$(y_k)_{k\in \N} \in M^{\N}$ converging to $y\in M$, such that $y_k\in J_{t,T}^+(x_k)$ for all $k\in \N$.  Then $y\in
\overline{J_{t,T}^+(x)}$.  \end{lemme}

\begin{proof} For $k\in \N$, consider a future curve $\gamma_k :  [0,1] \rightarrow M$ parametrized by arc length, of length $\ell_k$ between $t$ et $T$, such that $\gamma_k(0)=x_k$ and $\gamma_k(1) = y_k$.  Up to the choice of a sub sequence, we can assume that $(\gamma_k)$ converges uniformly to a curve $\gamma$, and that $\ell_k$ converges to $\ell \in \intff{t}{T}$.  We have $\gamma(0) = x$ et $\gamma(1) = y$.  \\ 
\indent Let $\e>0$.  By Proposition \ref{long_curve}, there is a future curve $\eta_{\e}$ of length $\tilde{\ell}\in \intff{\ell- \e}{\ell +\e}$ joining $x$ and $y$.  We can either shorten or extend $\eta_{\e}$ to a future curve of length $\ell$ with endpoint $z_{\e}$ satisfying $z_{\e}\in J_{t,T}^+(x)$ and $d(y,z_{\e})\le \e$, therefore $y\in V_{\e}(J_{t,T}^+(x))$ for all $\e >0$ and $y\in \overline{J_{t,T}^+(x)}$.

\end{proof}

\begin{coro} \label{outer}For all $x\in M$, $t>0$, $T>t$ and $\e>0$, there is a neighborhood $V$ of $x$ such that for all $y\in V$, $J_{t,T}^+(y)$ lies in the $\e$-neighborhood of $J_{t,T}^+(x)$.  \end{coro}

\begin{proof} Let us assume that this statement is false, so that there exists $x\in M$, $t>0$, $T>t$, $\e>0$, a sequence $(x_k)$ converging to $x$ and a sequence $(y_k)$ such that $y_k\in J^+_{t,T}(x_k)$ and $d(y_k , J_{t,T}^+(x)) \ge \e$. Since $y_k \in \overline{B(x_k, T)} \subset \overline{B(x, T+1)}$ for $k$ sufficiently large, we can assume up to extraction that $(y_k)$ converges to $y\in M$ (let us recall that since the metric $h$ is complete, by the Hopf-Rinow Theorem, closed balls are compact).  The previous result states that $y\in \overline{J_{t,T}^+(x)}$, but $d(y,J_{t,T}^+(x))\ge
\e$, which is absurd.

\end{proof}

\begin{lemme} \label{inner} For all $x\in M$, $t>0$, $T>t$ et $\e>0$, there is a neighborhood $V$ of $x$ such that for
all $y\in V$, $J_{t,T}^+(x) \subset V_{\e}(J_{t,T}^+(y))$.  \end{lemme}

\begin{proof} Let us once again prove this result by contradiction (it allows us to consider one point of $J_{t,T}^+(x)$
instead of the whole set).  Let us assume that there is a sequence $(x_k)$ converging to $x$, $\e>0$ and a sequence $(y_k)$
such that $y_k \in J_{t,T}^+(x)$ and $y_k\notin V_\e(J_{t,T}^+(x_k))$ for all $k$.  Let $(\gamma_k)$ be a sequence of future curves
with length between $t$ and $T$ such that $\gamma_k(0)=x$ and $\gamma_k(1)=y_k$.  Up to extraction, we can assume that
$(\gamma_k)$ converges to a future curve $\gamma$ and that $(\ell_h(\gamma_k))$ converges to $\ell \in [t,T]$.  Let $y$ be
the limit of $y_k$.  By proposition \ref{long_curve} there is a future curve parametrized by arc length $\eta$ joining $x$
and $y$ with length between $t-\frac{\e}{4}$ and $T+\frac{\e}{4}$.  Let $\nu_t$ be a time dependent locally Lipschitz
everywhere causal vector field with constant norm such that $\nu_t(\eta(t))=\dot{\eta}(t)$.  Let $\varphi_t$ denote the
isotopy of $\nu_t$.  The map $\varphi_1$ is continuous, therefore $\varphi_1(x_k)$ converges to $\varphi_1(x)=y$.  Therefore
$y\in V_{\frac{\e}{4}}(J^+_{t-\frac{\e}{4},T+\frac{\e}{4}}(x_k)) \subset V_{\frac{\e}{2}}(J_{t,T}^+(x_k))$ for $k$ large
enough, and $y_k\in V_{\e}(J^+_{t,T}(x_k))$, which is absurd.

\end{proof}

By combining Corollary \ref{outer} and Lemma \ref{inner}, we obtain the following result.

\begin{theo} \label{continuity_adapted} Let $h$ be an adapted Riemannian metric and let $0<t<T$.  The map $x\mapsto \overline{J_{t,T}^+}(x)$ is
continuous with respect to the Hausdorff topology.\end{theo}

Theorem \ref{continuity_adapted} along with the existence of adapted metrics (Proposition \ref{adapted_exists}) proves Theorem \ref{theo:continuity_future}.

\section{Attractors in spacetimes}
\label{sec:attractors}

\subsection{Pre attractors, attractors and basin of attraction}

\indent Let us recall the definition of attractors.\\

\noindent \textbf{Definition \ref{defi_attractor}} \emph{An open set $U\subset M$ is said to be a pre attractor if there is  $t_0>0$ such that $\overline{J^+_{t_0}(U)} \subset U$.\\ The set $A=\bigcap_{t\ge t_0}\overline{J^+_t(U)}$ is called the attractor.\\ The set $B(A,U)=\bigcup_{t\ge 0} \{p\in M \vert \overline{J^+_t(p)}\subset U\}$ is called the basin of $U$-attraction.\\ If $\mathcal U$ is the set of pre attractors sharing the same attractor $A$, then the basin of attraction is  $B(A)=\bigcup_{U\in\mathcal U}B(A,U)$.}\\

%\indent Note that the distinction between $B(A,U)$ and $B(A)$ has to be made: in the Minkowski space $\R^{1,n-1}$, the empty set is an attractors, with pre attractor $I^+(x)$ for any point $x\in \R^{1,n-1}$. In this case, we see that $B(\emptyset, I^+(x))=I^+(x)$, but $B(\emptyset)=\R^{1,n-1}$.\\
\indent Note that the attractor $A$ may be empty. The central result of this section is Theorem \ref{mainthm}, which is the equivalent in Lorentzian geometry of Conley's Theorem for flows. Let us recall its statement.\\

\noindent \textbf{Theorem \ref{mainthm}.} \emph{If $A$ is  an attractor in a spacetime $(M,g)$ (for an adapted Riemannian metric), then there is a $B(A)\setminus A$-time function.}\\

\indent The proof consists mainly in observing that attractors attract long curves that start in the basin of attraction. We start by working with $B(A,U)\setminus A$ for a given pre attractor $U$, then extend to $B(A)$. When  considering long curves, the choice of the Riemannian metric becomes important, which is why we will only consider adapted Riemannian metrics. \\
\indent Let us start with a lemma that justifies the name attractor.

\begin{lemme} \label{attractor} Let $U$ be a pre attractor and $A$ its attractor.  Let $(\gamma_i)_{i\in \N}$ be a sequence of future directed causal curves with length tending to $+\infty$ and such that $\gamma_i(0)\in U$.  Then for any neighborhood $V$ of $A$ and for any compact set $K$, if we have $z_i=\gamma_i(1)\in K$ for all $i$, then $z_i\in V$ for all $i$ sufficiently large.  \end{lemme}

\begin{proof} We will prove this result by contradiction.  Let us assume that we can extract a sub-sequence (still written $z_i$) such that $z_i \notin V$.  After a second extraction, we can assume that $z_i$ converges to a $z\in K$.\\
\indent Consider $i\in \N$.  If $j$ is large enough, then $\ell(\gamma_j)\ge \ell(\gamma_i)$, hence $z_j\in J^+_{\ell(\gamma_i)}(U)$.  Therefore $z\in \overline{J^+_{\ell(\gamma_i)}(U)}$ for all $i$, and $z\in A$, which gives us $z_i \in V$ for any large $i$, which is a contradiction with our hypothesis.  \end{proof}

This result means that long future directed curves that start in the pre attractor and that do not go to infinity will be attracted by $A$.  Another important fact is that the basin of attraction is open.

\begin{coro} \label{basin_open} Let $h$ be a Riemannian metric adapted to $g$, and let $U$ be a pre attractor and $A$ its attractor.  Then the basin of  $U$-attraction $B(A,U)$ is open.  \end{coro}

\begin{proof} Let $B_s=\{x\in M\vert J^+_s(x)\subset U\}$, so that $B(A,U)=\bigcup_{s>0}B_s$. Let $(x_k)_{k\in\N}$ be a sequence in $M\setminus B_s$. Assume that $x_k\to x\in M$. For all $k\in \N$, there is a future directed curve $\gamma_k:\intff{0}{1}\to M$ such that $\gamma_k(0)=x_k$ and $\gamma_k(1)\notin U$, of length $\ell_h(\gamma_k)> s$. Since $J^+_T(U)\subset U$, we can also assume that $\ell_h(\gamma_k)\le s+T$, therefore, up to changes of parameters, we can assume that $\gamma_k$ converges towards a future directed curve $\gamma$. According to Proposition \ref{long_curve}, we can find a future curve $\eta$ such that $\eta(0)=x$ and $\eta(1)=\gamma(1)\notin U$ of length greater than $s$, which shows that $x\notin B_s$, i.e. $B_s$ is open, and so is $B(A,U)=\bigcup_{s>0}B_s$.
\end{proof}

%------------------------------------------------------------------------------------------------------------------------------
%------------------------------------------------------------------------------------------------------------------------------
%------------------------------------Construction-of-a-time-function-----------------------------------------------------------
%------------------------------------------------------------------------------------------------------------------------------
%------------------------------------------------------------------------------------------------------------------------------

\subsection{Construction of a $B(A)\setminus A$-time function}

We start with a refined version of Urysohn's Lemma adapted to attractors. Conley's construction of a Lyapunov function for an attractor in \cite{conley} associated to a flow $\p^t$ consists in considering the function $\sup_{t\ge 0} f(\p^t(x))$ where $f(x)=\frac{d(x,A)}{d(x,A)+d(x,B(A)^c)}$.  In the non compact case, Hurley noticed in  \cite{hurley_98} that one has to choose a different function $f$. 

\begin{lemme} \label{construction} Let $U$ be a pre attractor and $A$ its attractor.  Then there exists a continuous function $f:  M \rightarrow [0,1]$ such that:
\begin{enumerate}\item $f^{-1}(0)=A$ \item $f^{-1}(1)=M\setminus B(A,U)$ \item For all $x\in B(A,U)$, there is a   neighborhood $N$ of $x$ such that: $$\forall \varepsilon>0 ~\exists t_0>0~ \forall y\in J^+_{t_0}(N)~~f(y)<\varepsilon$$  \end{enumerate} \end{lemme}

\begin{proof} We will first look for a function $\Psi$ such that the following function will almost satisfy our conditions:
\begin{displaymath} f(x) = \frac{d(x,A)}{d(x,A) + \Psi(x) d(x,M\setminus B(A,U))} \end{displaymath}

\noindent \textbf{First step:}  Construction of $\Psi$\\
 \indent We start by writing $M=\bigcup_{n\in \mathbb{N}} K_n$ where the $K_n$ are compact and where for all $n$, $K_n$ lies in the interior of $K_{n+1}$.\\ \indent Let $T>0$ be such that $\overline{J^+_T(U)}\subset U$. The goal is to obtain a continuous function $\Psi:  M \rightarrow [1,+\infty[$ such that $\Psi(x)d(x,M\setminus B(A,U))\ge n$ on $\overline{J^+_T(U)} \cap M\setminus K_n$\\
\indent For $i\in \N$, we denote by $U_i$ a relatively compact open set of $M$, such that $\overline{U_i}\subset B(A,U)$ and $(K_i\setminus K_{i-1}) \cap\overline{J^+_T(U)}\subset U_i$, and $U_i\cap K_{i-2} = \emptyset$ (e.g. an $\e$ neighborhood of $(K_i\setminus K_{i-1})\cap \overline{J^+_T(U)}$ with $\e$ small enough).\\ \indent We now consider $U_{\infty}=M\setminus \overline{J^+_T(U)}$ so that we have an open cover $M=\bigcup_{i\in \mathbb{N}\cup\{\infty\}}U_i$, and let $\theta_i$ be a partition of unity associated to this open cover.\\
 \indent Finally, we consider $\alpha_i=\min(\frac{i}{\inf_{U_i}d(.,M\setminus B(A,U))},1)$ and $\alpha_{\infty}=1$, and let $\Psi = \sum \theta_i \alpha_i$. Since $\alpha_i\ge 1$ for all $i$, we have $\psi \ge 1$, and $\psi$ is continuous (even smooth) because the sum is locally finite.\\
\indent Let $x\in \overline{J^+_T(U)} \cap M\setminus K_n$. If $x\in U_i$, then $x\in K_i$, therefore $i > n$. We have:

\begin{eqnarray*} \psi(x)d(x,M\setminus B(A,U)) &=&d(x,M\setminus B(A,U)) \sum_{x\in U_i} \theta_i(x) \alpha_i \\ &\ge& d(x,M\setminus B(A,U))\sum_{x\in U_i} \theta_i(x) \frac{i}{\inf_{U_i}d(.,M\setminus B(A,U))} \\ &\ge &  \sum_{x\in U_i}i \theta_i (x) \underbrace{ \frac{d(x,M\setminus B(A,U))}{\inf_{U_i}d(.,M\setminus B(A,U))}}_{\ge 1} \\&\ge & \sum_{x\in U_i}i \theta_i(x) >  n\sum_{x\in U_i} \theta_i(x) =n \end{eqnarray*}

\noindent \textbf{Second step:}  Construction of $f$\\
\indent Define $\mu:M\to \R$ by  $\mu(x)=\min(1,d(x,A))$ (set $\mu(x)= 1$ if $A=\emptyset$)  and consider the function $f:M\to \R$ such that: $$f(x)=\frac{\mu(x)}{\mu(x)+\Psi(x) d(x,M\setminus B(A,U))}$$ \indent  If $B(A,U)=M$, then we set $f(x)=\frac{\mu(x)}{\mu(x)+\Psi(x)}$. The function $f$  is continuous, has values in $[0,1]$ and satisfies the two first requirements.\\

\noindent \textbf{Third step:}  Checking the last requirement\\
\indent Let $x\in B(A,U)$.  We know that $I^-(x)\cap B(A,U) \ne \emptyset$ since $B(A,U)$ is open and $x$ lies in the closure of $I^-(x)$.  Let $y\in I^-(x)\cap B(A,U)$, and let $N$ be a compact neighborhood of $x$ included in $I^-(x)\cap B(A,U)$.\\
\indent Since $y\in B(A,U)$, there is $t_1>0$ such that $J^+_{t_1}(y)\subset U$, therefore $J^+_{t_1}(N)\subset U$ and $\overline{J^+_{t_2}(N)}\subset \overline{J^+_T(U)}$ where $t_2=t_1+T$.\\
\indent Let $\varepsilon>0$.  Because of Lemma \ref{attractor}, if we consider $n>\frac{1}{\varepsilon }$ and $V$ the $\e$-neighborhood of $A$, then there is $t_0>t_2$ such that $J^+_{t_0}(N)\subset V\cup (M\setminus K_n)$ (if it was not the case, we could construct a sequence $\gamma_i$ of curves with length growing to infinity and with endpoints in $K_n$ but not in $V$, which would be a contradiction).\\ \indent Therefore if $y\in J^+_{t_0}(N)$, then either $y\in V$,
in which case $M(y)<\varepsilon$ results in $f(y)<\varepsilon$, either $y\notin K_n$ in which case $f(y)\le \frac{1}{0+n}<\varepsilon$.  \end{proof}

We will now use this function to construct a $B(A,U)\setminus A$-time function.

\begin{lemme} \label{g_t_continuous} Let us consider an attractor $A$ and $f$ the function given by Lemma  \ref{construction}.  For $t\ge 0$ we consider:  \begin{displaymath} g_t(x)=\sup_{J^+_t(x)}f \end{displaymath} The  function $g_t$ is continuous.  \end{lemme}

\begin{proof} The idea is to see that locally, we can find $t'>t$ such that $g_t(x)=\max_{\overline{J^+_{t,t'}(x)}} f$, and use the continuity of the map $x\mapsto \overline{J^+_{t,t'}(x)}$ to conclude.\\
\indent  Let us start by considering the case where $x\in B(A,U)$.  If $g_t(x)>0$, then let $U$ be a small compact neighborhood of $x$, and set: 
$$C=\min_{y\in U} \max_{z\in \overline{J^+_{t,t+1}}(y)} f(z)$$
 \indent  Let us show that if $U$ is small enough, then $C>0$.  If not, then we can find a sequence $x_k\to x$ such that $\max_{ \overline{J^+_{t,t+1}}(x_k)} f = 0$, therefore by continuity of $z\to \overline{J_{t,t+1}^+}(z)$, we have $\max_{\overline{J_{t,t+1}^+}(x)}f=0$ and $\overline{J_{t,t+1}^+}(x) \subset A$, therefore $\overline{J^+_t}(x)\subset A$ and $g_t(x)=0$, which is absurd.  We choose $U$ such that $C>0$.  According to Lemma \ref{construction}, there is a neighborhood $N$ of $x$ and $t_0>0$ such that $f(y)\le \frac{C}{2}$ for all $y\in J^+_{t_0}(N)$. Therefore, for $y\in N \cap U$, we have $g_t(y)=\max_{\overline{J_{t,t_0}^+}(y)} f$ which is a continuous function, and $g_t$ is continuous at $x$.\\
 \indent If $g_t(x)=0$, let $\e >0$.  By Lemma \ref{construction}, there is a neighborhood
$N$ of $x$ and $t_0>t$ such that $f(y)\le \e$ for $y\in J^+_{t_0}(N)$.  Let $W$ be a neighborhood of $x$ such that $\max_{\overline{J_{t,t_0}^+}(y)} f <\e$ for $y\in W$ (recall that this map is continuous and has value $0$ at $x$). For $y\in W\cap N$, we have $g_t(y)\le \e$, therefore $g_t$ is continuous at $x$.\\ \indent Let us now consider the case where $x\notin B(A,U)$.  First, let us show that $g_t(x)=1$. If $g_t(x)<1$, then $\overline{J^+_t}(x) \subset f^{-1}(\intfo{0}{1})=B(A,U)$. Consider the set $E$ of endpoints of future causal curves starting at $x$ of length $t$.  Then $E$ is relatively compact and $\overline{E} \subset \overline{J^+_t}(x) \subset B(A,U)$.  Since $B(A,U) = \bigcup_{s\ge 0} B_s$ where $B_s =\{ x\in M\vert J^+_s(x)\subset U\}$ is open (see the proof of Corollary \ref{basin_open}), we have a finite cover $E\subset \bigcup_{1\le i \le k} B_{t_i}$.  Set $t'=t+ \max t_i$, we find that $J^+_{t'}(x)\subset U$ and $x\in B(A,U)$ which is absurd.  Therefore $g_t(x)=1$.\\
\indent Finally, consider $T>t$ such that $g_t(x)=\sup_{J_{t,T}^+(x)}f$.  Let $\e>0$ and let $U$ be a neighborhood of $x$ such that $\max_{\overline{J_{t,T}^+}(y)}f > 1-\e$ for $y \in U$.  We have $g_t(y)\ge \max_{\overline{J_{t,T}^+}(y)}f \ge 1-\e$ for $y\in U$, therefore $g_t$ is continuous at $x$.

\end{proof}

Let us see how we can obtain a $B(A,U)\setminus A$-time function.

\begin{prop} \label{decreasing} Let $A$ be an attractor.  Let $x\in B(A,U)\setminus A$ and $y\in J^+(x)\setminus \{x\}$, there is an interval $I$ of real numbers with non empty interior such that $g_s(y)<g_s(x)$ for all $s\in I$.
\end{prop}

\begin{proof} \indent Let $t>0$ such that $y\in J^+_t(x)$.  Let $\e = \frac{f(x)}{2}$ and consider:
\begin{displaymath} a = \inf \{ u >0 / ~g_u(x)  \le \e \}  \end{displaymath}
\indent The continuity of $f$ implies that $a>0$. Let $I =\intoo{\max(a-t, 0)}{a}$. If  $s\in I$, then $g_s(x) > \e$ and  $s+t > a$, therefore $g_s(y) \le g_{s+t}(x) \le g_a(x)\le \e < g_s(x)$, and $g_s(y)<g_s(x)$.
\end{proof}

\begin{coro} \label{conclusion} Let $A$ be an attractor.  Consider $(u_q)_{q\in \Q_+^*}$ a sequence of positive real numbers such that $\sum_{q\in \Q_+^*} u_q = 1$.  Then $\tau_A=\sum_{q\in \Q_+^*} u_q (1-g_q)$ is a $B(A,U)\setminus A$-time function.  \end{coro}

\begin{proof} The function $\tau_A$ is continuous because the $g_q$ are continuous and the sum converges normally.  Let $x\in M$, and let us consider $y\in J^+(x)$.  For $t\ge 0$, we have $J^+_t(y) \subset J_t^+(x)$, which gives us $g_t(y) \le g_t(x)$, and $\tau_A(y) \ge \tau_A(x)$.  If $x\in B(A,U)\setminus A$ and $y\ne x$, then Proposition \ref{decreasing} gives us an interval with non empty interior $I$ such that $g_t(y) < g_t(x)$ for $t\in I$.  Let $q_0 \in I \cap \Q$.
\begin{eqnarray*} \tau_A(y) &=& u_{q_0} (1- g_{q_0}(y)) + \sum_{q\ne q_0} u_q (1-g_q(y)) \\ &\le& u_{q_0} (1-
g_{q_0}(y)) + \sum_{q\ne q_0} u_q (1-g_q(x)) \\ &<& u_{q_0} (1- g_{q_0}(x)) + \sum_{q\ne q_0} u_q (1-g_q(x)) \\ &<&
\tau_A(x) \end{eqnarray*} \end{proof}

We can easily extend this to the basin $B(A)$.

\begin{proof}[Proof of Theorem \ref{mainthm}]  Since $B(A)$ is a subset of $M$, it is a separable metric space and therefore it satisfies the Lindel\"of property: of any open cover we can extract a countable cover. This allows us to choose a sequence of pre attractors $(U_n)_{n\in \N } $ with attractor $A$  such that $B(A)=\bigcup_{n\in \N} B(A,U_n)$. For $n\in \N$, let $\tau_n$ be a $B(A,U_n)\setminus A$-time function.  The function $\tau = \sum_{n\in \N} 2^{-n} \tau_n$ is a $B(A)\setminus A$-time function.
\end{proof}

The same technique provides a  time function for the union of all basins of attraction.
\begin{coro} \label{all_attractors} Let $(M,g)$ be a spacetime and let $\mathcal{A}$ be the set of all attractors. There is a $\bigcup_{A\in \mathcal{A}}B(A,U)\setminus A$-time function. \end{coro}

\begin{proof} Let $U=\bigcup_{A\in \mathcal{A}}B(A)\setminus A$. By the Lindel\"of property, we can choose a sequence of attractors $(A_n)_{n\in \N } \in \mathcal{A}^{\N}$ such that $U=\bigcup_{n\in \N} B(A_n)\setminus A_n$. For $n\in \N$, let $\tau_n$ be a $B(A_n)\setminus A_n$-time function. The function $\tau = \sum_{n\in \N} 2^{-n} \tau_n$ is a $U$-time function.

\end{proof}

\section{Chain recurrence, and a Lorentzian Conley Theorem}
\label{sec:chain_recurrence}
Closed future curves are an obvious obstruction to the existence of a time function. Classically, the existence of a time function is linked to stable causality, which has an inconvenient: its definition involves other metrics, whereas the existence of a time function does not. We will see that we can find an obstruction to the existence of a time function that does not involve  nearby metrics: chain recurrence. The idea of chains  consists in joining points by sequences of long curves and small jumps from the end of a curve to the beginning of the next.\\
\indent Let us recall the definition of chain recurrence.\\

\noindent \textbf{Definition \ref{def_chain_rec}.} \emph{Let $\e \in \PP(M)$, $T>0$ and $p,q\in M$.\\ An $(\e,T)$-chain from $p$ to $q$ is a finite sequence of future directed causal curves $(\gamma_i:[0,1]\to M)_{i=1,\dots ,k}$ of length at least $T$ such that:} 
\begin{enumerate} \item $d(p,\gamma_1(0))\le \e( p)$
\item $d(\gamma_i(1),\gamma_{i+1}(0))\le \e(\gamma_i(1))$ for all $i<k$
\item $d(\gamma_k(1),q)\le \e(q)$
\end{enumerate}

\emph{A point $p\in M$ is said to be chain recurrent if for any $\e \in \PP(M)$ and $T>0$ there is an $(\e,T)$-chain from $p$ to $p$. We will denote by $R(g)$ the set of chain recurrent points.}\\

Conley noticed that chain recurrence is linked to attractors. The same goes for spacetimes.

\begin{prop} Let $x\notin R(g)$. There exists an attractor $A$ such that $x\in B(A)\setminus A$.
\end{prop}

\begin{proof} Since $x\notin R(g)$, let us consider $\e \in \PP(M)$ and $T>0$ such that there is no $(\e,T)$-chain from $x$ to $x$. Let $U$ be the set of points $y\in M$ such that there is an $(\e,T)$-chain from $x$ to $y$. It follows from the definition of $(\e,T)$-chains that $U$ is open. We will start by proving that $U$ is a pre attractor.\\
\indent Let $y\in \overline{J^+_T(U)}$. Let us consider a function $\delta \in \PP(M)$ such that $\delta \le \frac{\e}{2}$ and that $d(y,z)<\delta (z)$ implies $\e(z)>\frac{\e(y)}{2}$ (the existence of such a function, which can be seen as continuous continuity modulus of $\e$, is proved in \cite{Choi:2002}). Let $z\in B(y,\delta(y))\cap J^+_T(U)$. We can write $z=\gamma(1)$ where $\gamma$ is a future curve with length at least $T$ and such that $\gamma(0)\in U$. Since there is an $(\e,T)$-chain $\gamma_1,\dots,\gamma_k$ from $x$ to $\gamma(0)$ (by definition of $U$), the choice of the function $\delta$ was made in such a way that $\gamma_1,\dots,\gamma_k,\gamma$ is an $(\e,T)$-chain from $x$ to $y$, therefore $y\in U$.\\
\indent If $\gamma$ is a future curve of length at least $T$ such that $\gamma(0)=x$, then $\gamma(1)\in U$ (because $\gamma$ itself is an $(\e,T)$-chain). Therefore $J^+_T(x)\subset U$, and $x\in B(A,U)$. \\
\indent Since there is no $(\e,T)$-chain from $x$ to $x$, we know that $x\notin U$, but $A\subset U$, hence $x\notin A$. We have shown that $x\in B(A,U)\setminus A\subset B(A)\setminus A$.
\end{proof}

By combining this result and Corollary \ref{all_attractors}, we obtain the first result mentioned in this paper:\\

\noindent \textbf{Theorem \ref{chain_recurrent_partial_time}.} \emph{Let $(M,g)$ be a spacetime. There is a $M\setminus R(g)$-time function. Particularly, if $R(g)=\emptyset$, then there is a time function.}\\

As mentioned earlier, chain recurrence is an obstruction to the existence of a time function, therefore the last statement of this theorem is an equivalence.

\begin{theo} \label{time_no_chains} Let $(M,g)$ be a spacetime that admits a time function. Then for any $T>0$ there is a function $\e \in \PP(M)$ such that there is no $(\e,T)$-chain with same end points, and therefore $R(g)=\emptyset$.
\end{theo}

\begin{proof} Let $f$ be a time function. If $K\subset M$ is compact and $T>0$ let $\alpha_{K,T}=\inf \{ \vert f(y)-f(x) \vert / y\in K \textrm{ and } y\in J^+_T(x) \}$.\\
\indent Since $h$ is adapted to $g$, we have $\overline{J^+_T(x)}\subset J^+_{T/ \sqrt 2}(x)$, which shows that $\alpha_{K,T}>0$.\\
\indent Let us fix $x_0\in M$ and write $M=\bigcup_{n\in \N} K_n$ where $K_n=\overline{B}(x_0,nT)$. For $x\in M$ and $T>0$, we will denote by $n(x)$ the smallest integer $n$ such that $x\in \mathring{K}_n$.\\
\indent We will construct a function $\e \in \PP(M)$ such that for all $x\in M$ and for all $y\in B(x,\e(x))$, we have $f(y)\le f(x)+\frac{1}{2}\alpha_{K_{n(x)},T}$.\\

\indent Let us consider $x\in M$ and $U_x$ a relatively compact open neighborhood of $x$ that lies in $\mathring{K}_{n(x)}$. For $y\in U_x$, we have $n(y)\le n(x)$ hence $\alpha_{K_{n(y)}} \ge \alpha _{K_{n(x)}}$. The compactness of  $\overline{U_x}$ and the continuity of $f$ assure the existence of  $\delta_x>0$ such that for all $y\in U_x$ and $z\in B(y,\delta_x)$, we have $f(z) \le f(y) + \frac{\alpha_{K_{n(x)}}}{2} \le f(y) + \frac{\alpha_{K_{n(y)}}}{2}$. From the open covering $M=\bigcup_{x\in M} U_x$, we extract a locally finite covering $M = \bigcup_{i\in I} U_{x_i}$. Let $\e_i = \inf \{ \delta_j / ~U_{x_i}\cap U_{x_j} \ne \emptyset \}$ and let  $(\theta_i)_{i\in I}$ be a partition of unity subordinate to  $M = \bigcup_{i\in I} U_{x_i}$. The function $\e = \sum_{i\in I} \e_i \theta_i$ meets our requirements.\\

\indent Let $x\in M$ and consider an $(\e,T)$-chain from $x$ to $x$. It can be seen as a sequence of points $(x_1,x_2, \dots, x_k, y_1, y_2,\dots, y_{k-1})$ in $M$ such that $x_1=x=x_k$, $y_i \in J^+_T(x_i)$ and $d(y_i,x_{i+1})<\e(y_i)$.\\
\indent We have $f(x_{i+1}) \le f(y_i) + \frac{\alpha_{K_{n(y_i)}}}{2}$, but $\vert f(y_i) - f(x_i) \vert \le \alpha_{K_{n(y_i)}}$, therefore $f(x_{i+1})\le f(x_i) -  \frac{\alpha_{K_{n(y_i)}}}{2} < f(x_i)$, which implies $f(x_k) < f(x_1)$, i.e. $f(x) < f(x)$, which is absurd.\\

\indent We have shown that for any $x\in M$, there is no $(\e,T)$-chain from $x$ to $x$.

\end{proof}

\section{Time functions for stably causal spacetimes}
\label{sec:stably_causal}

Hawking's Theorem states that  a spacetime admits a time function if and only if it is stably causal. Although the important result is the existence of a time function for a stably causal spacetime, the necessity of stable causality is non trivial, and was not proved by Hawking. So far, it seems that the only available proof of this necessity is to first use Bernal and Sanchez's Theorem from \cite{BS_fonction_lisse} that shows the existence of a  temporal function, and it is easy to see that a temporal function is a time function for all close metrics. The problem in the topological case is that a time function is not necessarily a time function for close metrics (consider for example a linear function with past directed lightlike gradient in Minkowski space).\\
\indent In the previous section, we showed that the absence of chain recurrence is equivalent to the existence of a (continuous) time function. It would be interesting to a find a proof of that a time function implies stable causality without using differentiable functions, and a possibility would be to show that the absence of chain recurrence implies stable causality.\\
\indent We are now going to prove the direct sense in Hawking's Theorem: stable causality implies the existence of a time function. By using Corollary \ref{all_attractors} all we have to see is that the set $\bigcup_{A\in \mathcal{A}} B(A)\setminus A$ is the whole manifold $M$.

\begin{lemme} \label{future_pre_attractor} Let $g'\succ g$ and $x\in M$. The chronological future $U=I^+_{g'}(x)$ of $x$ for $g'$ is a pre attractor.
\end{lemme}
\begin{proof} We will show that  $\overline{J^+_1(U)}\subset U$. If $y\in \overline{J^+_1(U)}$, we can find a sequence $\gamma_k$ of past directed causal curves (for $g$), with unit speed (for $h$)  such that $\gamma_k(0) \to y$ and $\gamma_k(t_k)\in U$ for some $t_k\ge 1$. Since a causal curve for $g$ is also causal for $g'$, we have $\gamma_k(1)\in U$. Up to a sub sequence, we can assume that $\gamma_{k/[0,1]}$ converges uniformly to a past directed causal curve $\gamma$. Let $z=\gamma(1)$. We have $z\in J^-_g(y)$, therefore $z\in I^-_{g'}(y)$. Since $I^-_{g'}(y)$ is open and $\gamma_k(1)\to z$, we have $\gamma_k(1) \in I^-_{g'}(y)$ for $k$ large enough, and $\gamma_k(1) \in U$ implies $y\in U$.
\end{proof}

Note that Lemma \ref{future_pre_attractor} is valid regardless of causality, but in general the point $x$ may lie in the attractor associated to $I^+_{g'}(x)$ (it is the case if there is a closed causal curve passing through $x$). In order to show that the point $x$ does not lie in the attractor, the right condition is strong causality.\\
\indent We will use Proposition \ref{stable_implies_strong} which ensures that stably causal spacetimes are strongly causal. In the classical proof of Hawking's Theorem, a slightly weaker notion is used (distinguishing spacetimes), which would probably also be sufficient in our situation.\\

\noindent\textbf{Proposition \ref{stable_attractors}.} \emph{Let $(M,g)$ be a stably causal spacetime. Then for all $x\in M$, there is an attractor $A$ such that $x\in B(A)\setminus A$.}\\

\begin{proof}  Let $x\in M$ and let $g'\succ g$ be a strongly causal metric (this is possible because $g$ is stably causal, see Proposition \ref{stable_implies_strong}). Let $W$ be a neighborhood of $x$ in $U$ such that the intersection of any $g'$-causal curve with $W$ is connected, small enough to satisfy Proposition \ref{short_curves} (there is an upper bound on the length of $g$-causal curves in $W$), and let us consider $z\in I^-_g(x)\cap W$. Then $U=I^+_{g'}(z)$ is a pre attractor, and $x\in U\subset B(A,U)$, where $A$ is the attractor associated to $U$. \\
\indent Let us show that $x \notin A$.  If it were the case, then for all $t>0$ there would be a $g$-causal future curve $\gamma_t$ of length at least $t$ such that $\gamma_t(0)\in U$ and $d(\gamma_t(1),x)\le 1/t$. Since $\gamma_t(0)\in U=I^+_{g'}(z)$, we can consider a $g'$-causal future curve $\eta_t$ such that $\eta_t(0)=z$ and $\eta_t(1)=\gamma_t(1)$ (the concatenation of a $g'$-timelike curve from $z$ to $\gamma_t(0)$ and of $\gamma_t$). If $t$ is large enough, then $\eta_t(1)\in W$, and therefore $\eta_t(s)\in W$ for all $s\in [0,1]$, hence $\gamma_t(s)\in W$ for all $s\in [0,1]$, but this is impossible because $\ell_h(\gamma_t)\to +\infty$. Therefore $x\notin A$.
\end{proof}

 To complete the proof of the direct sense in  Hawking's Theorem, notice that according to Proposition \ref{stable_attractors}, the union $\bigcup_{A\in \mathcal A} B(A)\setminus A$, where $\mathcal A$ is the set of attractors, is equal to the whole manifold $M$.  Corollary \ref{all_attractors} implies that there is time function.\\
 \indent Finally, the proof of Theorem \ref{stable_no_chains} follows by adding in Theorem \ref{time_no_chains}.

~\\
\footnotesize \textsc{Université du Luxembourg, Campus Kirchberg, 6, rue Richard Coudenhove-Kalergi, L-1359 Luxembourg}\\
 \emph{E-mail address:}  \verb|daniel.monclair@uni.lu|
\end{document}